\documentclass[12pt]{amsart}%
\usepackage{amsfonts}
\usepackage{txfonts}
\usepackage{mathtools}
\usepackage{amsmath}
\usepackage{amssymb}
\usepackage{amsthm}
\usepackage{newlfont}
\usepackage{graphicx}%
\providecommand{\U}[1]{\protect\rule{.1in}{.1in}}
\newtheorem{thm}{Theorem}[section]

\newtheorem{cor}[thm]{Corollary}
\newtheorem{lem}[thm]{Lemma}
\newtheorem{prop}[thm]{Proposition}

\theoremstyle{definition}

\newtheorem{rem}[thm]{Remark}
\numberwithin{equation}{section}

\newcommand{\R}{\mathbb R}

\newcommand{\ed}{\end {document}}

\DeclareMathOperator{\ric}{Ric}
\begin{document}
\title{Yau Type Gradient Estimates For $\Delta u + au(\log u)^{p}+bu=0$ On Riemannian Manifolds}
\author{Bo Peng}
\address{School of Mathematical Sciences, UCAS, Beijing 100190, China; Institute of Mathematics
The Academy of Mathematics and systems of sciences,Chinese academy of sciences}
\email{pengbo17@mails.ucas.ac.cn}

\author{Youde  Wang}
\address{1. College of Mathematics and Information Sciences, Guangzhou University; 2. Hua Loo-Keng Key Laboratory of Mathematics, Institute of Mathematics, Academy of Mathematics and Systems Science, Chinese Academy of Sciences, Beijing 100190, China; 3. School of Mathematical Sciences, University of Chinese Academy of Sciences, Beijing 100049, China.}
\email{wyd@math.ac.cn}

\author{Guodong Wei}
\address{
School of Mathematics (Zhuhai), Sun Yat-sen University, Zhuhai, Guangdong 519082, P. R. China}
\email{weigd3@mail.sysu.edu.cn}\thanks{ }

\begin{abstract}
In this paper, we consider the gradient estimates of the positive solutions to the following equation defined on a complete Riemannian manifold $(M, g)$
$$\Delta u + au(\log u)^{p}+bu=0,$$
where $a, b\in \mathbb{R}$ and $p$ is a rational number with $p=\frac{k_1}{2k_2+1}\geq2$ where $k_1$ and $k_2$ are positive integer numbers. we obtain the gradient bound of a positive solution to the equation which does not depend on the bounds of the solution and the Laplacian of the distance function on $(M, g)$. Our results can be viewed as a natural extension of Yau's estimates on positive harmonic function.

\end{abstract}
\maketitle

\section{Introduction}
Recently, one payed attention to studying the following elliptic equation defined on a complete, noncompact Riemannian manifold $(M, g)$
\begin{equation}\label{equ*}
\Delta u(x) + a(x)u(x)^\beta\left(\log u(x)\right)^{\alpha}+b(x)u(x)=0,
\end{equation}
where $\alpha, \beta\in\mathbb{R}$ and $a(x), b(x)\in C^2(M)$. The equation has some relations to the geometrical quantities. On one hand, it is linked with gradient Ricci solitons, for example, see \cite{CL, C-C*, Ma, Yang} for detailed explanations. On the other hand, it is closely related to log-Sobolev constants of Riemmannian manifolds (see \cite{G, Chu-Y}). Throughout this article, we will use the notation $\ric(g)$ to denote the Ricci curvature of $(M,g)$. Recall that, log-Sobolev constants $S_M$, associated to a closed Riemannian manifold $(M, g)$, are the smallest positive constants such that the logarithmic-Sobolev inequality
$$\int_Mu^2\log u^2dM\leq S_M\int_M|\nabla u|^2dM$$
holds for all smooth function $u$ defined on $M$ with $\int_Mu^2dM=\mbox{Vol}(M)$. Chung and S.-T. Yau \cite{Chu-Y} showed that if the function $\psi$ attains the log-Sobolev constant $S_M$ with $\int_M\psi^2dM=\mbox{Vol}(M)$, then it must satisfy
$$\Delta \psi + S_M \psi\log \psi^2 = 0.$$
They then showed that
$$\sup_{x\in M}\psi \leq e^{n/2}, \quad\quad |\nabla\psi|^2 +2S_M \psi^2\log\psi\leq S_Mn\psi^2$$
and used these estimates to give a lower bound for $S_M$ if $\ric(g)\geq 0$. Moreover, F. Wang extends this to the case $\ric(g)\geq -K$ in \cite{Wf}. For more details we refer to \cite{Chu-Y, Wf}.

In physics, the logarithmic Schr\"odinger equation written by
$$i\epsilon\frac{\partial\Psi}{\partial t}=-\epsilon^2\Delta\Psi+(W(x)+w)\Psi -\Psi\log|\Psi|^2, \quad\Psi:[0, \infty)\times\mathbb{R}^n\to\mathbb{C},\,\, n\geq1,$$
has also received considerable attention. It is well-known that this class of equation has some important physical applications, such as quantum mechanics, quantum optics, open quantum systems, effective quantum gravity, transport and diffusion phenomena, theory of superfluidity and Bose-Einstein condensation (see \cite{Z} and the references therein). In its turn, standing waves solution, $\Psi$, for this logarithmic Schr\"odinger equation is related to the solutions of the following equation
$$-\epsilon^2\Delta u + V(x)u -u\log u^2=0,$$
where $V(x)$ is a real function on $\mathbb{R}^n$. Many mathematicians have also ever studied the existence and properties of solutions to such elliptic equation on an Euclidean spaces (see \cite{A-J, W-Z} and references therein).

From the viewpoint of analysis, for the case $b(x)\equiv0$ and $\alpha=0$ the equation (\ref{equ*}) reduces to $\Delta u+ a(x)u^\beta=0$ which has been studied intensively. For instance, for the family of equations
$$\Delta u + u^\beta=0,\quad\quad n/(n -2) \leq \beta < (n+2)/(n-2)$$
with subcritical Sobolev growth, let $u\geq 0$ be a nonnegative solution of the equation in $B_1(0)\setminus\{0\}\subset\mathbb{R}^n$ with a nonremovable isolated singularity. Gidas and Spruck \cite{BJ} used analytic techniques to prove that, for $n/(n-2) <\beta< (n+2)/(n-2)$,
$$u = (1 + o( 1 ))c_0/|x|^{2/(\beta-1)}\quad\quad \text{as}\quad  x\to 0,$$
where $c_0=c(n, \beta)$ is a positive constant number. Moreover, we are reminiscent to that the general equations $\Delta u + g(u)=0$ on $B_1(0)\setminus\{0\}\subset\mathbb{R}^n$ were also considered in \cite{CGS}, where $g(u)$ is nondecreasing function for $u$, $g(0)=0$ and satisfies some technical conditions. In particular, They proved that for $\beta=n/(n-2)$, any nonnegative solution $u$ to the above equation with a nonremovable isolated singularity behaves like
$$u=(1+o(1))\left[\frac{(n-2)^2/2}{r^2\log(1/r)}\right]^{(n-2)/2},$$
which gives an improvement of a result of Aviles \cite{A}.  Recently, M. Ghergu, S. Kim and H. Shahgholian \cite{GKS} studied the semilinear elliptic equation
$$\Delta u + u^\beta(|\log u|)^\alpha=0, \quad\quad\text{on}\quad B_1(0)\setminus\{0\},$$
where $B_1\subset\mathbb{R}^n$ with $n\geq3$, $n/(n-2) < \beta < (n+2)/(n-2)$ and $-\infty<\alpha<+\infty$. They showed that nonnegative solution $u\in C^2(B_1\setminus\{0\})$ of the above equation either has a removable singularity at the origin or behaves like some class of functions as $x\to 0$. They extended the classical argument in \cite{CGS, BJ} mentioned in the above to a log-type nonlinearity. For more heuristical details, we refer to \cite{A, CGS, BJ, PWW} and references therein.

In order to focus on the core of the problem and not to lengthen this article by adding technicalities, we restrict ourselves to the case of $b(x)\equiv constant$ and $\beta=1$ in this paper. That is, we focus on studying the gradient estimate of the positive solution to the following nonlinear elliptic equation defined on an $n$-dimensional complete noncompact Riemannian manifold $(M,g)$
\begin{equation}\label{equ}
\Delta u(x) + au(x)\left(\log u(x)\right)^{p}+bu(x)=0,
\end{equation}
where $a\neq 0$, $b$ are two constants and $$p=\frac{k_{1}}{2k_{2}+1} \geq 2,$$ here $k_{1}$ and $k_{2}$ are two positive integers.

Now, let's recall some previous work related closely to this paper. In the case $a(x)\equiv0$ and $b(x)\equiv0$, \eqref{equ*} is the Laplace equation. The corresponding  gradient estimate of \eqref{equ*} has ever been established by Yau in the very famous paper \cite{yau} (for its generalized version, see the remarkable work \cite{C-Y} due to Cheng-Yau). In 1980s, for the case $a(x)\equiv0$ and $b(x)\neq0$ is a smooth function P. Li and S.T. Yau \cite{LY} proved the well-known Li-Yau estimate for the corresponding heat equation and derived a Harnack inequality. The importance of gradient estimates and Harnack inequalities can not be overemphasized in geometric analysis and mathematical physics. There is a huge literature on the gradient estimates to the solutions of some elliptic and parabolic equations. In general, these estimates have been used to find H\"olders' continuity of solutions, sharp estimate on the fundamental solution, estimate on the principal eigenvalue. Here we would like to refer but a few to \cite{CH, Chow, C-C, BJ, KZ, LY, Ma, SY, JS, Wang} and references therein.

For the case $a<0$ is a constant, $b=0$, $\beta=1$ and $\alpha=1$ in (\ref{equ*}), that is,
\[
\Delta u(x) +au(x) \log u(x) = 0 \quad \mbox{on} \quad M,
\]
Ma \cite{Ma} studied the gradient estimates of the positive solutions to the above elliptic equation in the case $\dim(M)\geq 3$. Yang \cite{Yang} considered the following
$$\Delta u(x) +au(x) \log u(x) + bu= 0 \quad \mbox{on} \quad M$$
where $a$ and $b$ are two real numbers, and improves the estimate of \cite{Ma} and extends it to the case $a>0$ and $M$ is of any dimension. L. Chen and W. Chen \cite{C-C*} also studied independently the elliptic equation and improves the estimate of \cite{Ma} and extends it to the case $a>0$. In \cite{Yang} and \cite{C-C*} they also studied  $$u_t=\Delta u(x) +au(x) \log u(x) + bu \quad \mbox{on} \quad M$$
and derive a local gradient estimate for the positive solution of the parabolic equation defined on complete noncompact manifolds with a fixed metric and curvature locally bounded below. Later, Cao et al \cite{CL}, H. Dung and N. Dung \cite{Dung}, Huang and Ma\cite{HM}, Qian \cite{Q*}, Zhu and Li \cite{Z-L} also studied the gradient estimates of the positive solutions to the above elliptic or the corresponding parabolic equation in the case $a\in\mathbb{R}$.

On the other hand, in \cite{Ab}, the author also considered the following
\[
\Delta_f u(x) +au(x) (\log u(x))^\alpha = 0
\]
defined on a complete smooth metric measure manifold with weight $e^{-f}$ and Bakry-\'Emery Ricci tensor bounded from below, where $a$ and $\alpha$ are real constants. He obtained the local gradient estimates, which depends on the bound of positive solutions, and prove the global gradient estimates on bounded positive solutions to the equation.

While its parabolic counterpart,
\[
\left( \Delta - q(x,t) - \frac{\partial}{\partial t} \right)u(x,t) = a u(x,t)(\log (u(x,t)))^{\alpha},
\]
where $q(x,t)$ is a $C^{2}$ function, $a$ and $\alpha$ ($\alpha\neq 0$ or $1$) are constants, was also considered by some mathematicians(see \cite{C-C*, Yang, Z-L}). Wu \cite{W} and Yang and Zhang \cite{Y-Z} also paid attention to a similar nonlinear parabolic equation defined on some kind of smooth metric measure space.

In the present paper, for the case $p=\frac{k_{1}}{2k_{2}+1} \geq 2$, we try to improve the classical methods and employ some delicate analytic techniques to obtain a gradient bound of a positive solution to \eqref{equ} which does not depend on such quantities as the bounds of the solution and the Laplacian of the distance function. Now we are in the position to state the main results of this paper.
\begin{thm}[\bf Local gradient estimate]\label{main}
Let $\left( M,g\right)$ be an $n$-dimensional complete Riemannian manifold. Suppose there exists a nonnegative constant $K:=K(2R)$ such that $\ric(g) \geq -Kg$ in the geodesic ball $\mathbb{B}_{2R}(O)$ where $O$ is a fixed point on $M$. Then, for any smooth positive solution $u(x)$ to equation \eqref{equ} with $a \neq 0$ and $p=\frac{k_{1}}{2k_{2}+1} \geq 2$ where $k_{1}$ and $k_{2}$ are integers and any $1<\lambda<2$, there holds true that on $\mathbb{B}_{R}(O)$,
\begin{itemize}
\item [(1)] if $a > 0$,
\begin{equation*}
\begin{split}
\frac{\left\vert \nabla u \right\vert^{2}}{u^{2}}+\lambda a \big(\log u\big)^{p}\leq \,\,  & n \bigg[ 
\frac{\lambda}{2-\lambda}\frac{\left((n-1)(1+\sqrt{K}R)+2\right)C_{1}^{2}+C_{2}}{R^{2}}\\
&+\max\left(\frac{\lambda^{2}}{4(\lambda-1)^{2}},\frac{2\lambda^{2}}
{(2-\lambda)^{2}}\right)\frac{nC_{1}^{2}}{R^{2}}\\
&+\max\left(\frac{2\lambda}{\lambda-1},\frac{2\lambda}{2-\lambda}
\right)K +C(n,p,a, b,\lambda)\bigg], 
\end{split}
\end{equation*}
where $C(n,p,a, b,\lambda)$ can be expressed explicitly as
\begin{equation*}
\begin{split}
C(n,p,a, b,\lambda)= &\max\left(\frac{6\lambda}{2-\lambda},\frac{16-7\lambda}{\lambda-1},\frac{16-7\lambda}{2-\lambda}\right)\frac{2\left\vert b\right\vert}{n} +\frac{9(\lambda-1)aH^{p+1}}{2n(p-1)}\\
& +\frac{3(4-\lambda)}{2\lambda n(p-1)}H\left\vert b\right\vert+\max(3(\lambda-1),\lambda)\frac{2a}{n\left\vert J \right\vert^{p}}+\lambda (2n(p-1))^{p-1}a;
\end{split}
\end{equation*}
\item[(2)] if $a < 0$,
\begin{equation*}
\begin{split}
\frac{\left\vert \nabla u \right\vert^{2}}{u^{2}}+\lambda a \left(\log u\right)^{p}\leq \,\, & n\bigg[ \frac{\lambda}{2-\lambda}\frac{\left((n-1)(1+\sqrt{K}R)+2\right)C_{1}^{2}+C_{2}}{R^{2}}\\ &+\max\left(\frac{\lambda^{2}}{4(\lambda-1)^{2}},\frac{\lambda^{2}}{(2-\lambda)^{2}}\right)\frac{nC_{1}^{2}}{R^{2}}\\
&+2\max\left(\frac{\lambda}{\lambda-1},\frac{\lambda}{2-\lambda}\right)K+C(n,p,a, b,\lambda)\bigg],
\end{split}
\end{equation*}
with
\begin{equation*}
\begin{split}
C(n,p,a, b,\lambda)=&\max\left(\frac{6\lambda}{2-\lambda},\frac{16-7\lambda}{2-\lambda}\right)\frac{2\left\vert b\right\vert}{n}-\frac{3\lambda\left\vert L \right\vert^{p}}{2n}a+\frac{3(4-\lambda)\left\vert b\right\vert \left\vert L\right\vert}{2\lambda n(p-1)}\\
&-\frac{4(\lambda-1)aV^{p}}{n}.
\end{split}
\end{equation*}
\end{itemize}
Here, $C_1$ and $C_2$ are two positive constants independent of the goemetry of $(M,g)$,
\[
J =\frac{(2-\lambda)p-\sqrt{(2-\lambda)^{2}p^{2}+\frac{8\lambda(\lambda-1)}{n}p(p-1)}}{2\lambda p(p-1)}, \quad \quad V = \frac{pn+\sqrt{n^{2}p^{2}+8np(p-1)}}{4}\frac{\lambda}{\lambda-1},
\]
\[
L=\frac{(\lambda-4)np-\sqrt{(\lambda-4)^{2}n^{2}p^{2}+32\lambda(\lambda-1)np(p-1)}}{8(\lambda-1)} \quad \text{and} \quad H=\max\left(\left\vert L\right\vert,V\right).
\]
\end{thm}

\medskip
Actually for any $\lambda >1$, we can obtain similar results to Theorem \ref{main}.
\begin{thm} \label{any}
Suppose that $(M, g)$ satisfies the same conditions as in Theorem \ref{main}. If $u$ is a smooth positive solution $u(x)$ to equation \eqref{equ} with $a \neq 0$ and $p=\frac{k_{1}}{2k_{2}+1} \geq 2$ where $k_{1}$ and $k_{2}$ are integers. Then for any $2\leq\lambda<+\infty$ there holds true that, on $\mathbb{B}_{R}(O)$, if $a>0$, then
\[
\frac{\left\vert \nabla u \right\vert^{2}}{u^{2}} +\lambda a(\log u)^{p}\leq 2\lambda n \bigg( 2K+A+\frac{6nC_{1}^{2}}{R^{2}}+\frac{(2n(p-1))^{p-1}}{2}a+\frac{12\left\vert b\right\vert}{n}+\frac{a}{n\left\vert J \right\vert^{p}}+\frac{V(aV^{p}+\left\vert b\right\vert)}{n(p-1)}\bigg);
\]
and if $a<0$, then
\[
\frac{\left\vert \nabla u \right\vert^{2}}{u^{2}} +\lambda a(\log u)^{p}\leq 2\lambda n \bigg( 2K+A+\frac{6nC_{1}^{2}}{R^{2}}+\frac{12\left\vert b\right\vert}{n}-\frac{2aV^p}{n}+\frac{|b|V}{n(p-1)}\bigg).
\]
Here $$A=\frac{\left((n-1)(1+\sqrt{K}R)+2\right)C_{1}^{2}+C_{2}}{R^{2}},$$ 
$J$ and $V$ are constants defined in Theorem \ref{main} with $\lambda=3/2$.
\end{thm}

Another consequence of Theorem \ref{main} is the following Harnack inequality:
\begin{cor}[\bf Harnack inequality]\label{har}
Suppose the same conditions as in Theorem \ref{main} hold. Then, for $a > 0$ and $p=\frac{2k}{2k_{2}+1} \geq 2$ where $k$ and $k_{2}$ are integers there holds true
\[
\sup_{\mathbb{B}_{R/2}(O)}u \leq e^{R\sqrt{3n \left( 2K+A+\frac{6nC_{1}^{2}}{R^{2}}+\frac{(2n(p-1))^{p-1}}{2}a+\frac{12\left\vert b\right\vert}{n}+\frac{a}{n\left\vert J \right\vert^{p}}+\frac{V^{p+1}a}{n(p-1)}+\frac{V\left\vert b\right\vert}{n(p-1)} \right)}}\inf_{\mathbb{B}_{R/2}(O)}u.
\]
Here $A$ is the same as in the above, $J$ and $V$ are constants defined in Theorem \ref{main} with $\lambda=3/2$.
\end{cor}
\medskip

Moreover, we have the following results which are analogous to Qian's results in \cite{Q*} where Qian obtained a uniform estimates of positive solutions from the global gradient estimates for the corresponding nonlinear heat equations, which was proved by modifying slightly the method of Yau \cite{Yau}.
\begin{cor} \label{mor}
Let $\left( M,g\right)$ be an $n$-dimensional complete noncompact Riemannian manifold. Suppose there exists a nonnegative constant $K$ such that $\ric(g) \geq -Kg$ on $M$.  Then, for any smooth positive solution $u(x)$ to equation \eqref{equ}  with $a > 0$ and $p=\frac{k_{1}}{2k_{2}+1} \geq 2$ where $k_{1}$ and $k_{2}$ are integers, there holds true
\begin{align*}
\big(\log u\big)^{p}\leq &\frac{2n}{a} \left( 2K +\frac{12\left\vert b\right\vert}{n}+\frac{V\left\vert b\right\vert}{n(p-1)} \right) + \left(n^p(2(p-1))^{p-1}+\frac{2}{\left\vert J \right\vert^{p}}+\frac{2V^{p+1}}{(p-1)}\right).
\end{align*}
Here $J$ and $V$ are two constants defined in Theorem \ref{main} with $\lambda=3/2$.

In particular, if $a>0$ and $p=\frac{k_1}{2k_2+1}\geq 2$, then
\begin{align*}
u\leq &\exp\left\{\left(\frac{2n}{a} \left( 2K+\frac{12\left\vert b\right\vert}{n}+\frac{V\left\vert b\right\vert}{n(p-1)} \right) + \left(n^p(2(p-1))^{p-1}+\frac{2}{\left\vert J \right\vert^{p}}+\frac{2V^{p+1}}{(p-1)}\right)\right)^{\frac{1}{p}}\right\};
\end{align*}
and if $a>0$ and $p=\frac{2k}{2k_2+1}\geq 2$, the following holds
\begin{align*}
|\log u|\leq &\left[\frac{2n}{a} \left( 2K+\frac{12\left\vert b\right\vert}{n}+\frac{V\left\vert b\right\vert}{n(p-1)} \right) + \left(n^p(2(p-1))^{p-1}+\frac{2}{\left\vert J \right\vert^{p}}+\frac{2V^{p+1}}{(p-1)}\right)\right]^{\frac{1}{p}}.
\end{align*}
\end{cor}

\begin{rem} If $a=b=0$. By letting $\lambda=\frac{3}{2}$ in Theorem \ref{main}, we obtain
\begin{equation*}
\frac{\left\vert \nabla u \right\vert^{2}}{u^{2}} \leq  3n \left( 2K+A+\frac{6nC_{1}^{2}}{R^{2}}\right).
\end{equation*}
Plugging the expression of $A$ into the above inequality, we arrive at
\begin{equation*}
\frac{\left\vert \nabla u \right\vert^{2}}{u^{2}} \leq  3n \left( 2K+ \frac{\left((n-1)(1+\sqrt{K}R)+2(3n+1)\right)C_{1}^{2}+C_{2}}{R^{2}}\right).
\end{equation*}
This is an estimates which is of the form of Yau's estimate. So, our method can be regarded as an extension of Yau's method.
\end{rem}

\begin{rem} In Theorem \ref{main}, the assumption $p=k_{1}/(2k_{2}+1)$ where $k_{1}$ and $k_{2}$ are two positive integers, is a natural condition since $\log u$ may be negative. Otherwise, $(\log u)^p$ does not make sense. In fact, it is not difficult to find that we can also prove some similar results with those in \cite{C-C*, Ma}. However, it seems that for the case $1< p < 2$ the method employed here is not valid. It is worth to point out that the method adopted here is also effective for the case $p=1$.
\end{rem}

The paper is organized as follows. In Section 2, we establish some basic lemmas which will be highly used to prove main results in this paper. The proof of main theorems and related corollaries is provided in Section 3.

\section{Preliminaries}
Throughout this section, we will denote by $\left( M,g\right)$ an $n$-dimensional complete Riemannian manifold with $\ric(g) \geq -Kg$ in the geodesic ball $\mathbb{B}_{2R}(O)$, where $K=K(2R) $ is a nonnegative constant depending on $R$ and $O$ is a fixed point on $M$. First, we consider the following equation on $M$
\begin{equation}\label{aug}
\Delta u + auf(\log u) +bu=0,
\end{equation}
where $f\in C^{2}(\R,\R)$ is a $C^2$ function and $a \neq 0$. It is easy to see that if we set $f(t)=t^p$, then  the equation \eqref{aug} is just the equation \eqref{equ}.

\begin{prop}\label{prop} Let $\left( M,g\right)$ be an $n$-dimensional complete Riemannian manifold satisfy the same assumption as in Theorem \ref{main}. Suppose that $u(x)$ is a smooth positive solution to equation \eqref{aug} on geodesic ball $\mathbb{B}_{R}(O)$ and let
$$ \omega=\log u\quad \text{and}\quad  G=\left\vert \nabla\omega \right\vert^{2}+\beta_{1} f(\omega)+\beta_{2},$$
here $\beta_1$ and $\beta_2$ are two constants to be determined later. Then we have
\begin{align*}
\Delta G \geq\, & \frac{2}{n}G^{2} +\bigg( (\beta_{1} - 2a)f^{'}(\omega)+\beta_{1} f^{''}(\omega)-\frac{4}{n}(\beta_{1} -a)f(\omega)-2K-\frac{4}{n}(\beta_{2}-b)\bigg)G+2K\beta_{2}\\
&-2\left<\nabla\omega,\nabla G\right> + \frac{2}{n}(\beta_{1} -a)^{2}f^{2}(\omega)-\beta_{1} (\beta_{1} -a)f(\omega)f^{'}(\omega)-\beta_{1}^{2}f(\omega)f^{''}(\omega)\\
&+\bigg(2K \beta_{1}+\frac{4}{n}(\beta_{1}-a)(\beta_{2}-b)\bigg) f(\omega) +(2a\beta_{2}-b\beta_{1}-\beta_{1}\beta_{2})f^{'}(\omega)-\beta_{1}\beta_{2}f^{''}(\omega).
\end{align*}
\end{prop}

\begin{proof}
First, we notice that there hold
\begin{equation}\label{1}
\Delta\omega+G+(a-\beta_{1})f(\omega)+b-\beta_{2}=0,
\end{equation}
and
\begin{equation}\label{2}
\left\vert \nabla\omega \right\vert^{2}=G-\beta_{1} f(\omega)-\beta_{2}.
\end{equation}
By the Bochner-Weitzenb\"{o}ck$'$s formula and the assumption on the Ricci curvature tensor on $(M, g)$, we obtain
\begin{equation}\label{3}
\Delta \left\vert \nabla\omega \right\vert^{2} \geq 2\left\vert \nabla^{2}\omega \right\vert^{2}+2\left<\nabla\omega,\nabla(\Delta\omega)\right>-2K\left\vert \nabla\omega \right\vert^{2}.
\end{equation}
Combining \eqref{1}, \eqref{2} and \eqref{3}, we obtain
\begin{align}
\Delta G=\,&\Delta \left\vert \nabla\omega \right\vert^{2}+\Delta \left(\beta_{1} f(\omega)+\beta_{2}\right)\nonumber\\
\geq \,& 2\left\vert \nabla^{2}\omega \right\vert^{2}+2\left<\nabla\omega,\nabla(\Delta\omega)\right>-2K\left\vert \nabla\omega \right\vert^{2}+\Delta (\beta_{1} f(\omega))\nonumber\\
\geq \,& \frac{2}{n}(\Delta\omega)^{2}+2\left<\nabla\omega,\nabla(\Delta\omega)\right>-2K\left\vert \nabla\omega \right\vert^{2}+\beta_{1}\big(f^{''}(\omega)\left\vert \nabla\omega \right\vert^{2}+f^{'}(\omega)\Delta\omega\big).\nonumber
\end{align}
Here we have used the relation
\[
\left\vert \nabla^{2}\omega \right\vert^{2} \geq \frac{1}{n}(\Delta\omega)^{2}
\]
which can be easily derived by Cauchy-Schwarz inequality. Hence,
\begin{align}\label{re}
\Delta G\geq\, & \frac{2}{n}\bigg(G-(\beta_{1}-a)f(\omega)-(\beta_{2}-b)\bigg)^{2}-2\left<\nabla\omega,\nabla G \right>-2(a-\beta_{1})f^{'}(\omega)\left\vert \nabla\omega \right\vert^{2}\nonumber\\
& - 2K\left\vert \nabla\omega \right\vert^{2}+\beta_{1} f^{''}(\omega)\left\vert \nabla\omega \right\vert^{2}+\beta_{1} f^{'}(\omega)\bigg(-G+(\beta_{1}-a)f(\omega)+\beta_{2}-b\bigg)\nonumber\\
=\, & \frac{2}{n}\bigg(G-(\beta_{1}-a)f(\omega)-(\beta_{2}-b)\bigg)^{2}-2\left<\nabla\omega,\nabla G \right>-\beta_{1} f^{'}(\omega)G+\beta_{1}(\beta_{2}-b)f^{'}(\omega)\nonumber\\
&\, +\beta_{1}(\beta_{1}-a)f^{'}(\omega)f(\omega)+ \bigg( 2(\beta_{1} -a)f^{'}(\omega)+\beta_{1} f^{''}(\omega)-2K\bigg)\bigg(G-\beta_{1} f(\omega)-\beta_{2}\bigg)\nonumber\\
\geq\, & \frac{2}{n}G^{2}-2\left<\nabla\omega,\nabla G \right>+\bigg( (\beta_{1}-2a)f^{'}(\omega)+\beta_{1} f^{''}(\omega)-\frac{4}{n}(\beta_{1}-a)f(\omega)-\frac{4}{n}(\beta_{2}-b)-2K\bigg)G\nonumber\\
& + \frac{2}{n}(\beta_{1} -a)^{2}f^{2}(\omega)-\beta_{1} (\beta_{1} -a)f(\omega)f^{'}(\omega)-\beta_{1}^{2}f(\omega)f^{''}(\omega)-\beta_{1}\beta_{2}f^{''}(\omega)+2K\beta_{2}\nonumber\\
& +\bigg(2K \beta_{1}+\frac{4}{n}(\beta_{1}-a)(\beta_{2}-b)\bigg) f(\omega)+(2a\beta_{2}-b\beta_{1}-\beta_{1}\beta_{2})f^{'}(\omega).
\end{align}
This is just the required inequality. Thus we complete the proof.
\end{proof}

Next, we will turn to construct the cut-off function. Let $\psi(r)\in C^{2}\left([0,\infty),\R_{\geq0}\right)$ be a nonnegative $C^{2}$  function on $[0,\infty)$ such that $\psi(r)=1$ for $r \leq 1$, $\psi(r) =0$ for $r \geq 2$, and $0 \leq \psi(r) \leq 1$. Furthermore we can arrange that  $\psi(r)$ satisfying the following
\[
0 \geq \psi^{'}(r) \geq -C_{1}\psi^{\frac{1}{2}}(r) \quad \text{and} \quad \psi^{''}(r) \geq -C_{2}
\]
for some absolute constants $C_{1}$ and $C_{2}$. Now, let
\begin{equation}\label{cut}
\phi(x)=\psi\left(\frac{d(x,O)}{R}\right).\end{equation}
It is easy to see that
$$\phi(x)\big|_{\mathbb{B}_{R}(O)}=1\quad\quad \text{and}\quad\quad \phi(x)\big|_{M\backslash\mathbb{B}_{2R}(O)}=0.$$
Moreover, by using Calabi's trick (see \cite{Calabi}), we can assume without loss of generality that the function $\phi$ is smooth in $\mathbb{B}_{2R}(O)$. Then, by the Laplacian comparison theorem (e.g., see \cite[P.20, Theorem 1.53]{Au}), the following lemma holds obviously.

\begin{lem}\label{lap} Let $\phi$ be defined as above.  Then the following two inequalities hold
\begin{itemize}
\item [(i)]
$$\frac{\left\vert \nabla\phi \right\vert^{2}}{\phi} \leq \frac{C_{1}^{2}}{R^{2}},$$
\item [(ii)]
$$\Delta \phi \geq -\frac{(n-1)(1+\sqrt{K}R)C_{1}^{2}+C_{2}}{R^{2}}.$$
\end{itemize}
\end{lem}

Now, let $x_{0} \in \mathbb{B}_{2R}(O)$ such that
$$
Q=\phi G(x_0)=\sup_{\mathbb{B}_{2R}(O)}\phi G.
$$
We can further assume without loss of generality that $Q>0$, since otherwise Theorem \ref{main} holds trivially with $\beta_{1}=\lambda a$, $\beta_{2}=2\left\vert b\right\vert$ and $f=t^p$. Note that $x_0\notin \partial \mathbb{B}_{2R}(O)$. Thus,  at $x_0$, we have
\begin{equation}
\nabla(\phi G)(x_{0})=0\quad \text{and}  \quad \Delta (\phi G)(x_{0}) \leq 0.
\end{equation}
This  implies
\begin{equation}\label{max}
\phi\nabla G =-G \nabla\phi\quad \text{and} \quad \phi\Delta G + G\Delta\phi - 2G\frac{\left\vert \nabla\phi \right\vert^{2}}{\phi} \leq 0.\end{equation}
Combining Lemma \ref{lap} and \eqref{max} yields
\begin{equation}\label{4}
AG \geq \phi\Delta G,
\end{equation}
where
\[
A = \frac{\left((n-1)(1+\sqrt{K}R)+2\right)C_{1}^{2}+C_{2}}{R^{2}}.
\]

On the other hand, it is easy to see from \eqref{2} and \eqref{max} that at $x_0$,
\begin{align}\label{5}
-\left<\nabla\omega,\nabla G\right>\phi&=G\left<\nabla\omega,\nabla\phi\right>\nonumber\\
 & \geq -G\left\vert \nabla\phi \right\vert(G-\beta_{1} f(\omega)-\beta_{2})^{\frac{1}{2}}.
\end{align}

Now, by substituting \eqref{4} and \eqref{5} into \eqref{re}, we obtain
\begin{align}\label{6}
AG \geq &\frac{2}{n}\phi G^{2} +\bigg( (\beta_{1}-2a)f^{'}(\omega)+\beta_{1} f^{''}(\omega)-\frac{4}{n}(\beta_{1}-a)f(\omega)-\frac{4}{n}(\beta_{2}-b)-2K\bigg)\phi G\nonumber\\
 & +\left((2a\beta_{2}-b\beta_{1}-\beta_{1}\beta_{2})f^{'}(\omega)-\beta_{1}\beta_{2}f^{''}(\omega)+2K\beta_{2}\right)\phi-2G\left\vert \nabla\phi \right\vert(G-\beta_{1} f(\omega)-\beta_{2})^{\frac{1}{2}}\nonumber\\
 &+\left[\frac{2}{n}(\beta_{1} -a)^{2}f^{2}(\omega)-\beta_{1} (\beta_{1} -a)f(\omega)f^{'}(\omega)-\beta_{1}^{2}f(\omega)f^{''}(\omega)\right]\phi\nonumber\\
 &+\bigg(2K \beta_{1}+\frac{4}{n}(\beta_{1}-a)(\beta_{2}-b)\bigg) \phi f(\omega).
\end{align}

\section{Proof of Main Results}
In this section, we will give the proof of Theorem \ref{main}, Theorem \ref{any} and related corollaries.

\begin{proof}[\bf Proof of Theorem \ref{main}]

Letting $\beta_{1}=\lambda a$ (here $1<\lambda <2$ is a constant), $\beta_2=2|b|$ and $f(t)=t^p$ in \eqref{6}, we arrive at
\begin{align}\label{0}
AG \geq& \frac{2}{n}\phi G^{2} + \bigg( -\frac{4}{n}(\lambda -1)a\omega^{p}+(\lambda -2)pa\omega^{p-1}+\lambda a p(p-1)\omega^{p-2}-2K-\frac{12\left\vert b \right\vert}{n}\bigg)\phi G\nonumber\\
 &-2\left\vert \nabla\phi \right\vert G\left(G-\lambda a \omega^{p}\right)^{\frac{1}{2}}+\bigg( \frac{2}{n}(\lambda-1)^{2}\omega^{2p}-\lambda(\lambda-1)p\omega^{2p-1}-\lambda^{2}p(p-1)\omega^{2p-2}\bigg)a^{2}\phi\nonumber\\
 &+\left(2\lambda K+\frac{4(\lambda-1)(2\left\vert b \right\vert -b)}{n}\right)a\phi \omega^{p}+(2(2-\lambda)\left\vert b\right\vert-\lambda b)p\omega^{p-1}a\phi\nonumber\\
 &-2\lambda a\left\vert b \right\vert p(p-1)\omega^{p-2}\phi.
\end{align}
According to the sign of $a$, we need to consider the following two cases:

\medskip

{\bf Case $1$}: $a > 0$. In this case, we need to handle the following four possibilities:

$(1). \ \omega^{p} \geq 0 $ and $\omega \in(-\infty, L]\cup[V, \infty)$, where
\begin{align*}
L= \frac{(\lambda-4)np-\sqrt{(\lambda-4)^{2}n^{2}p^{2}+32\lambda(\lambda-1)np(p-1)}}{8(\lambda-1)}
\end{align*}
and
\begin{align*}
V = \frac{pn+\sqrt{n^{2}p^{2}+8np(p-1)}}{4}\frac{\lambda}{\lambda-1};
\end{align*}

$(2).\ \omega^{p} \geq 0 $ and $L < \omega < V$;

$(3).\ \omega^{p} < 0 $ and $\omega \leq \frac{1}{J}$, where
\[
J=\frac{(2-\lambda)p-\sqrt{(2-\lambda)^{2}p^{2}+\frac{8\lambda(\lambda-1)}{n}p(p-1)}}{2\lambda p(p-1)};
\]
and

$(4).\  \omega^{p} < 0 $ and $\frac{1}{J} < \omega < 0$.
\medskip

Next, we will discuss them case by case.

$(1)$. If $\omega^{p} \geq 0$ and $\omega \in(-\infty, L]\cup[V, \infty)$, then we can see easily that
\begin{equation}\label{7}
(\lambda -2)pa\omega^{p-1}+\lambda a p(p-1)\omega^{p-2} \geq (\lambda -2)pa\left\vert \omega \right\vert^{p-1},
\end{equation}

\begin{equation}\label{7'}
a\omega^{p}  \leq \frac{1}{\lambda}\left(\left\vert \nabla\omega \right\vert^{2}+\lambda a \omega^{p}+2\left\vert b\right\vert\right)  =  \frac{G}{\lambda},
\end{equation}
and
\begin{eqnarray}\label{8}
 & &\frac{2}{n}(\lambda-1)^{2}\omega^{2p}-\lambda(\lambda-1)p\omega^{2p-1}-\lambda^{2}p(p-1)\omega^{2p-2}\nonumber\\
 &=&  \left(\frac{2}{n}(\lambda-1)^{2}\omega^{2}-\lambda(\lambda-1)p\omega-\lambda^{2}p(p-1)\right)\omega^{2p-2}\geq 0.
 \end{eqnarray}
In addition, it is easy to check
\begin{eqnarray}
& &\frac{4(\lambda-1)\left\vert b \right\vert}{n} \omega^{2}+(2(2-\lambda)\left\vert b\right\vert-\lambda b)p\omega-2\lambda \left\vert b \right\vert p(p-1) \nonumber\\
&\geq & \bigg(\frac{4(\lambda-1)}{n} \omega^{2}+(4-3\lambda)p\omega-2\lambda p(p-1)\bigg)\left\vert b \right\vert \geq 0  \quad \mbox{for}\,\, \omega \geq V  \nonumber
\end{eqnarray}
and
\begin{eqnarray}
& &\frac{4(\lambda-1)\left\vert b \right\vert}{n} \omega^{2}+(2(2-\lambda)\left\vert b\right\vert-\lambda b)p\omega-2\lambda \left\vert b \right\vert p(p-1) \nonumber\\
&\geq & \bigg(\frac{4(\lambda-1)}{n} \omega^{2}+(4-\lambda)p\omega-2\lambda p(p-1)\bigg)\left\vert b \right\vert \geq 0 \quad \mbox{for} \,\, \omega \leq L. \nonumber
\end{eqnarray}
Therefore,
\begin{align}\label{new}
&\left(2\lambda K+\frac{4(\lambda-1)(2\left\vert b\right\vert-b)}{n}\right)a\omega^{p}+\left(2(2-\lambda)\left\vert b\right\vert-\lambda b\right)pa\omega^{p-1} -2\lambda a\left\vert b \right\vert p(p-1)\omega^{p-2}\nonumber\\
\geq&\left(\frac{4(\lambda-1)\left\vert b \right\vert}{n} \omega^{2}+(2(2-\lambda)\left\vert b\right\vert-\lambda b)p\omega-2\lambda \left\vert b \right\vert p(p-1)\right)a\omega^{p-2} \geq 0.
\end{align}
By substituting \eqref{7}, \eqref{7'}, \eqref{8} and \eqref{new} into \eqref{0},  we derive
\begin{align*}
AG
\geq\, & \frac{2}{n}\phi G^{2}  -\frac{4(\lambda-1)}{n\lambda}\phi G^{2}+(\lambda-2)ap\left\vert \omega \right\vert^{p-1}\phi G -\left(2K+\frac{12\left\vert b \right\vert}{n}\right)\phi G-2\left\vert \nabla\phi \right\vert G^{\frac{3}{2}}\\
\geq\, & \frac{2(2-\lambda)}{n\lambda}\phi G^{2} +(\lambda-2)ap\left\vert \omega \right\vert^{p-1}\phi G -\left(2K+\frac{12\left\vert b \right\vert}{n}\right)\phi G-2\left\vert \nabla\phi \right\vert G^{\frac{3}{2}}.
\end{align*}
Noting that
\[
0 \leq \frac{\left\vert\omega\right\vert^{p}}{G}=\frac{\omega^{p}}{\left\vert \nabla\omega \right\vert^{2}+\lambda a \omega^{p}+2\left\vert b\right\vert} \leq \frac{1}{\lambda a},
\]
we have
\begin{equation}
\begin{split}
AG\geq\, & \frac{2(2-\lambda)}{n\lambda}\phi G^{2} +(\lambda-2)ap\left( \frac{G}{\lambda a}\right)^{\frac{p-1}{p}}\phi G -\left(2K+\frac{12\left\vert b \right\vert}{n}\right)\phi G-2\left\vert \nabla\phi \right\vert G^{\frac{3}{2}}\\
=\, & \frac{2(2-\lambda)}{n\lambda}\phi G^{2} +\frac{(\lambda-2)p}{\lambda}\left(\lambda a\right)^{\frac{1}{p}} G^{\frac{2p-1}{p}}\phi-\left(2K+\frac{12\left\vert b \right\vert}{n}\right)\phi G-2\frac{\left\vert \nabla\phi \right\vert}{\phi^{\frac{1}{2}}}\phi^{\frac{1}{2}} G^{\frac{3}{2}}\\
\geq\, & \frac{2(2-\lambda)}{n\lambda}\phi G^{2} +\frac{(\lambda-2)p}{\lambda}\left(\lambda a\right)^{\frac{1}{p}} G^{\frac{2p-1}{p}}\phi-\left(2K+\frac{12\left\vert b \right\vert}{n}\right)\phi G-2\frac{C_{1}}{R}\phi^{\frac{1}{2}} G^{\frac{3}{2}}.
\end{split}
\end{equation}
Dividing both side of the last inequality by $G$, we obtain
\begin{align}\label{50}
A \geq \frac{2(2-\lambda)}{n\lambda}Q +\frac{(\lambda-2)p}{\lambda}\left(\lambda a\right)^{\frac{1}{p}} Q^{\frac{p-1}{p}}-\left(2K+\frac{12\left\vert b \right\vert}{n}\right)-2\frac{C_{1}}{R}Q^{\frac{1}{2}} .
\end{align}
By Young's inequality,
\begin{align}\label{60}
2\frac{C_{1}}{R}Q^{\frac{1}{2}} \leq\,\,& \frac{(2-\lambda)Q}{2n\lambda}+\frac{2n\lambda C_{1}^{2}}{(2-\lambda)R^{2}},
\end{align}
and
\begin{align}
\frac{(2-\lambda)p}{\lambda}\left(\lambda a\right)^{\frac{1}{p}} Q^{\frac{p-1}{p}} \leq\,\,& \frac{(2-\lambda)Q}{2n\lambda}+(2-\lambda)(2n(p-1))^{p-1}a.
\end{align}
Substituting the above two inequalities into \eqref{50}, we obtain
\begin{align}\label{12}
A \geq& \frac{2(2-\lambda)}{n\lambda}Q -\left(\frac{(2-\lambda)Q}{2n\lambda}+(2-\lambda)(2n(p-1))^{p-1}a\right)-\left(2K+\frac{12\left\vert b \right\vert}{n}\right)\nonumber\\
& -\left(\frac{(2-\lambda)Q}{2n\lambda}+\frac{2n\lambda C_{1}^{2}}{(2-\lambda)R^{2}}\right)\nonumber\\
 = &\frac{(2-\lambda)}{n\lambda}Q-(2-\lambda)(2n(p-1))^{p-1}a-2K-\frac{12\left\vert b \right\vert}{n}-\frac{2n\lambda C_{1}^{2}}{(2-\lambda)R^{2}}.
\end{align}
It follows
\[
Q \leq \frac{n\lambda}{(2-\lambda)}\left(2K+A+(2-\lambda)(2n(p-1))^{p-1}a+\frac{12\left\vert b \right\vert}{n}+\frac{2n\lambda C_{1}^{2}}{(2-\lambda)R^{2}}\right).
\]
Hence
\begin{equation}\label{20}
\begin{split}
\sup_{\mathbb{B}_{R}}G \leq \sup_{\mathbb{B}_{2R}}(\phi G)\leq \,\, & n\bigg(\frac{2\lambda}{2-\lambda}K+\frac{\lambda}{2-\lambda}A+\lambda (2n(p-1))^{p-1}a+\frac{12\lambda \left\vert b \right\vert}{(2-\lambda)n}\\
& +\frac{2n\lambda^{2}C_{1}^{2}}{(2-\lambda)^{2}R^{2}}\bigg).
\end{split}
\end{equation}

\medskip

$(2)$. Assume $\omega^{p} \geq 0$ and $L < \omega < V$. Denote $H=\max\{\left\vert L\right\vert, V\}$. Then it is easy to see that
\begin{align*}
-\frac{4}{n}(\lambda -1)a\omega^{p}+(\lambda -2)pa\omega^{p-1}+\frac{\lambda}{3} a p(p-1)\omega^{p-2} \geq & -\frac{4}{n}(\lambda -1)a\left\vert \omega\right\vert^{p}+(\lambda -2)pa\left\vert \omega\right\vert^{p-1}\\
\geq & -\frac{4}{n}(\lambda -1)aH^{p}+(\lambda -2)paH^{p-1}
\end{align*}
and
\begin{align*}
& \frac{2a}{n}(\lambda-1)^{2}\omega^{p+2}-\lambda(\lambda-1)ap\omega^{p+1}-\lambda^{2}ap(p-1)\omega^{p}+\bigg(2\lambda K+\frac{4(\lambda-1)(2\left\vert b\right\vert-b)}{n}\bigg)\omega^{2} \\
&+(2(2-\lambda)\left\vert b\right\vert-\lambda b)p\omega-2\lambda \left\vert b\right\vert p(p-1)\\
\geq &-\lambda(\lambda-1)ap\left\vert \omega\right\vert^{p+1}-\lambda^{2}ap(p-1)\left\vert\omega\right\vert^{p}+(2(2-\lambda)\left\vert b\right\vert-\lambda b)p\omega-2\lambda \left\vert b\right\vert p(p-1)\\
\geq & -\lambda(\lambda-1)apH^{p+1}-\lambda^{2}ap(p-1)H^{p}-(4-\lambda)p\left\vert b\right\vert H-2\lambda \left\vert b\right\vert p(p-1).
\end{align*}
Thus, by substituting the above two inequalityies into \eqref{0}, we get
\begin{align*}
AG \geq\, & \frac{2}{n}\phi G^{2} + \bigg( -\frac{4a}{n}(\lambda -1)\omega^{p}+(\lambda -2)pa\omega^{p-1}+\frac{\lambda a}{3}  p(p-1)\omega^{p-2}\bigg)\phi G-2\left\vert \nabla\phi \right\vert G^{\frac{3}{2}}\\
& + \frac{2\lambda a}{3}  p(p-1)\omega^{p-2}\phi G-\left(2K+\frac{12\left\vert b \right\vert}{n}\right)\phi G +a\omega^{p-2}\bigg[\frac{2a}{n}(\lambda-1)^{2}\omega^{p+2}\\
& -\lambda(\lambda-1)ap\omega^{p+1}-\lambda^{2}ap(p-1)\omega^{p}+\bigg(2\lambda K+\frac{4(\lambda-1)(2\left\vert b\right\vert-b)}{n}\bigg)\omega^{2}\\
& +(2(2-\lambda)\left\vert b\right\vert-\lambda b)p\omega-2\lambda \left\vert b\right\vert p(p-1)\bigg]\phi \\
\geq\, & \frac{2}{n}\phi G^{2} +\bigg(-\frac{4}{n}(\lambda -1)aH^{p}+(\lambda -2)paH^{p-1}\bigg)\phi G-2\left\vert \nabla\phi \right\vert G^{\frac{3}{2}}-\left(2K+\frac{12\left\vert b \right\vert}{n}\right)\phi G\\
& +a \omega^{p-2}\bigg[\frac{2\lambda}{3}p(p-1)\phi G -\bigg(\lambda(\lambda-1)apH^{p+1}+\lambda^{2}ap(p-1)H^{p}+(4-\lambda)p\left\vert b\right\vert H\\
& +2\lambda \left\vert b\right\vert p(p-1)\bigg)\phi\bigg].
\end{align*}

Now, if $$\frac{2\lambda}{3}p(p-1)\phi G -\bigg(\lambda(\lambda-1)apH^{p+1}+\lambda^{2}ap(p-1)H^{p}+(4-\lambda)p\left\vert b\right\vert H+2\lambda \left\vert b\right\vert p(p-1)\bigg)\phi \geq0,$$
then, it follows
\begin{equation}
\begin{split}
AG & \geq \frac{2}{n}\phi G^{2} +\bigg(-\frac{4}{n}(\lambda -1)aH^{p}+(\lambda -2)paH^{p-1}\bigg)\phi G-2\left\vert \nabla\phi \right\vert G^{\frac{3}{2}}-\left(2K+\frac{12\left\vert b \right\vert}{n}\right)\phi G.
\end{split}
\end{equation}
This implies
\[
A \geq \frac{2}{n}\phi G +\bigg(-\frac{4}{n}(\lambda -1)aH^{p}+(\lambda -2)paH^{p-1}\bigg)\phi -2\frac{C_{1}}{R} \phi^{\frac{1}{2}} G^{\frac{1}{2}}-\left(2K+\frac{12\left\vert b \right\vert}{n}\right)\phi  .
\]
Immediately, by the Young's inequality we infer from the above inequality
\begin{equation}\label{25}
\phi G \leq n\left(2K + A +\frac{4}{n}(\lambda -1)aH^{p}+(2-\lambda)paH^{p-1}+\frac{12\left\vert b \right\vert}{n}+\frac{nC_{1}^{2}}{R^{2}}\right).
\end{equation}

On the other hand, if $$\frac{2\lambda}{3}p(p-1)\phi G -\bigg(\lambda(\lambda-1)apH^{p+1}+\lambda^{2}ap(p-1)H^{p}+(4-\lambda)p\left\vert b\right\vert H+2\lambda \left\vert b\right\vert p(p-1)\bigg)\phi <0,$$
it is easy to see that
 \begin{equation}\label{26}
\phi G \leq \frac{3(\lambda-1)a}{2(p-1)}H^{p+1}+\frac{3\lambda a}{2}H^{p}+\frac{3(4-\lambda)\left\vert b\right\vert}{2\lambda (p-1)}H+3\left\vert b\right\vert.
\end{equation}
Noting that,
$$\frac{1}{H} \leq  \frac{4}{pn+\sqrt{p^{2}n^{2}+8np(p-1)}} \frac{\lambda-1}{\lambda} \leq \frac{2(\lambda-1)}{(p-1)\lambda},$$
and
$$\frac{1}{H} \leq  \frac{2(\lambda-1)}{\sqrt{2np(p-1)}\lambda}.$$
Then it is easy to see from \eqref{25} and \eqref{26} that
\begin{align}\label{58}
Q \leq n\left(2K + A + \frac{9(\lambda-1)}{2n(p-1)}aH^{p+1}+\frac{3(4-\lambda)\left\vert b\right\vert}{2\lambda n(p-1)}H+\frac{12\left\vert b \right\vert}{n}+\frac{nC_{1}^{2}}{R^{2}}\right).
\end{align}

\medskip

$(3)$. Asuume $\omega^{p} < 0$ and $\omega \leq \frac{1}{J}$, where
\[
J=\frac{(2-\lambda)p-\sqrt{(2-\lambda)^{2}p^{2}+\frac{8\lambda(\lambda-1)}{n}p(p-1)}}{2\lambda p(p-1)},
\]
In this case, the following holds obviously
\begin{equation}\label{y}
-\frac{4}{n}(\lambda -1)+(\lambda -2)p\omega^{-1}+\lambda  p(p-1)\omega^{-2} \leq -\frac{2}{n}(\lambda -1).
\end{equation}
Moreover, a simple calculation shows that
\begin{equation}\label{yy}
\frac{2}{n}(\lambda-1)^{2}-\lambda(\lambda-1)p\omega^{-1}-\lambda^{2}p(p-1)\omega^{-2} \geq 0,
\end{equation}
and
\begin{equation}\label{yyy}
\begin{split}
& 2\lambda K+\frac{4(\lambda -1)}{n}(2\left\vert b \right\vert-b)+(2(2-\lambda)\left\vert b\right\vert -\lambda b)p\omega^{-1}-2\lambda \left\vert b \right\vert p(p-1)\omega^{-2} \\
\leq  &\,\,  2\lambda K+\frac{12(\lambda-1)\left\vert b \right\vert}{n}+(4-\lambda)p\left\vert b\right\vert \left\vert J\right\vert.
\end{split}
\end{equation}
On the other hand, by Young's inequality, we have
\begin{equation}\label{22}
2\left\vert \nabla\phi \right\vert G\left(G-\lambda a \omega^{p}\right)^{\frac{1}{2}} \leq \frac{R\left\vert \nabla\phi \right\vert^{2}}{C_{1}\phi^{\frac{1}{2}}} G^{\frac{3}{2}}+\frac{C_{1}\phi^{\frac{1}{2}}G^{\frac{1}{2}}}{R}\left(G-\lambda a \omega^{p}\right).
\end{equation}
Substituting the above four inequalities \eqref{y}, \eqref{yy}, \eqref{yyy} and \eqref{22} into \eqref{0}, we get
\begin{align*}
AG \geq\, & \frac{2}{n}\phi G^{2} -\bigg(2K+\frac{12\left\vert b\right\vert}{n}\bigg)\phi G-\frac{R\left\vert \nabla\phi \right\vert^{2}}{C_{1}\phi^{\frac{1}{2}}} G^{\frac{3}{2}}-\frac{C_{1}\phi^{\frac{1}{2}}G^{\frac{3}{2}}}{R}\\
&+ \bigg[\bigg(-\frac{4(\lambda-1)}{n}+(\lambda-2)p\omega^{-1}+\lambda p(p-1)\omega^{-2}\bigg)\phi G +\frac{\lambda C_{1}}{R}\phi^{\frac{1}{2}}G^{\frac{1}{2}}\\
& +\bigg(2\lambda K+\frac{4(\lambda-1)}{n}(2\left\vert b \right\vert-b)+(2(2-\lambda)\left\vert b\right\vert -\lambda b)p\omega^{-1}-2\lambda \left\vert b \right\vert p(p-1)\omega^{-2}\bigg)\phi\\
&+\bigg(\frac{2}{n}(\lambda-1)^{2}-\lambda(\lambda-1)p\omega^{-1}-\lambda^{2}p(p-1)\omega^{-2}\bigg)a\omega^{p}\phi\bigg]a\omega^{p}\\
\geq\,&\frac{2}{n}\phi G^{2} -\bigg(2K+\frac{12\left\vert b\right\vert}{n}\bigg)\phi G-\frac{R\left\vert \nabla\phi \right\vert^{2}}{C_{1}\phi^{\frac{1}{2}}} G^{\frac{3}{2}}-\frac{C_{1}\phi^{\frac{1}{2}}G^{\frac{3}{2}}}{R}+ \bigg[ -\frac{2(\lambda-1)\phi G}{n} \\
& +\frac{\lambda C_{1}}{R}\phi^{\frac{1}{2}}G^{\frac{1}{2}}+\bigg(2\lambda K+\frac{12(\lambda-1)\left\vert b \right\vert}{n}+(4-\lambda)p\left\vert b\right\vert \left\vert J\right\vert\bigg)\phi\bigg]a\omega^{p}.
\end{align*}

Now, if $$-\frac{2(\lambda-1)\phi G}{n} +\frac{\lambda C_{1}}{R}\phi^{\frac{1}{2}}G^{\frac{1}{2}}+\bigg(2\lambda K+\frac{12(\lambda-1)\left\vert b \right\vert}{n}+(4-\lambda)p\left\vert b\right\vert \left\vert J\right\vert\bigg)\phi \leq0,$$
then, it follows
\begin{equation}
\begin{split}
AG & \geq \frac{2}{n}\phi G^{2} -\bigg(2K+\frac{12\left\vert b\right\vert}{n}\bigg)\phi G-\frac{R\left\vert \nabla\phi \right\vert^{2}}{C_{1}\phi^{\frac{1}{2}}} G^{\frac{3}{2}}-\frac{C_{1}\phi^{\frac{1}{2}}G^{\frac{3}{2}}}{R}.
\end{split}
\end{equation}
This implies
\[
A \geq \frac{2}{n}\phi G -\bigg(2K+\frac{12\left\vert b\right\vert}{n}\bigg)\phi -\frac{2C_{1}}{R}\phi^{\frac{1}{2}}G^{\frac{1}{2}}.
\]
In view of the Young's inequality we can see easily from the above that
\begin{equation}\label{15}
\phi G \leq n\left(2K+\frac{12\left\vert b\right\vert}{n}+ A +\frac{nC_{1}^{2}}{R^{2}}\right).
\end{equation}

If $$-\frac{2(\lambda-1)\phi G}{n} +\frac{\lambda C_{1}}{R}\phi^{\frac{1}{2}}G^{\frac{1}{2}}+\bigg(2\lambda K+\frac{12(\lambda-1)\left\vert b \right\vert}{n}+(4-\lambda)p\left\vert b\right\vert \left\vert J\right\vert\bigg)\phi >0,$$
Noting that
\[
\left\vert J\right\vert \leq \frac{2(\lambda-1)}{(2-\lambda)np},
\]
Thus, we have
\begin{equation}\label{16}
\phi G \leq 2n\left(\frac{n\lambda^{2}C_{1}^{2}}{8(\lambda-1)^{2}R^{2}}+\frac{\lambda}{\lambda-1}K+\frac{(16-7\lambda)\left\vert b\right\vert}{(2-\lambda)n}\right).
\end{equation}
Combining \eqref{15} and \eqref{16}, we conclude that
\begin{equation}\label{31}
\sup_{B_R}\left(\frac{\left\vert \nabla u \right\vert^{2}}{u^{2}}+\lambda a (\log u)^{p}+2\left\vert b\right\vert \right) \leq n\left(\frac{2\lambda K}{\lambda-1} + A +\frac{n\lambda^{2}C_{1}^{2}}{4(\lambda-1)^{2}R^{2}}+\frac{2(16-7\lambda)}{(2-\lambda)n}\left\vert b\right\vert\right).
\end{equation}

\medskip

$(4).$ Assume $\omega^{p} < 0$ and $\frac{1}{J} < \omega < 0$. Then
\begin{align*}
 &-\frac{4(\lambda-1)a}{n}\omega^{p}+(\lambda-2)ap\omega^{p-1}+\left(\lambda a p(p-1)+\frac{2(\lambda-1)a}{n\left\vert J\right\vert^{2}}\right)\omega^{p-2}\\
\geq \,\, & (\lambda-2)ap\left\vert \omega \right\vert^{p-1}-\lambda a p(p-1)\left\vert \omega \right\vert^{p-2}-\frac{2(\lambda-1)a\left\vert \omega\right\vert^{p-2}}{n\left\vert J\right\vert^{2}}\\
\geq \,\, & \frac{(\lambda-2)ap}{\left\vert J \right\vert^{p-1}}-\frac{\lambda ap(p-1)}{\left\vert J \right\vert^{p-2}}-\frac{2(\lambda-1)a}{n\left\vert J\right\vert^{p}}.
\end{align*}
And
\begin{align*}
&\frac{2(\lambda-1)^{2}a}{n}\omega^{p+2}-\lambda(\lambda-1)ap\omega^{p+1}-\lambda^{2}ap(p-1)\omega^{p}+\bigg(2\lambda K+\frac{4(\lambda-1)}{n}(2\left\vert b\right\vert-b)\bigg)\omega^{2} \\
& +(2(2-\lambda)\left\vert b\right\vert-\lambda b)p\omega-2\lambda \left\vert b\right\vert p(p-1)\\
\leq\,\, &\lambda^{2}ap(p-1)\left\vert \omega\right\vert^{p}+\bigg(2\lambda K+\frac{12(\lambda-1)\left\vert b\right\vert}{n}\bigg)\left\vert \omega\right\vert^{2} +(4-\lambda)p\left\vert b\right\vert \left\vert \omega\right\vert\\
\leq\,\, &\frac{\lambda^{2}ap(p-1)}{\left\vert J \right\vert^{p}}+\frac{2\lambda nK+12(\lambda-1)\left\vert b\right\vert}{n\left\vert J\right\vert^{2}} +\frac{(4-\lambda)p\left\vert b\right\vert}{ \left\vert J\right\vert}.
\end{align*}
Combining \eqref{0}, \eqref{22} and the above two inequalities, we obtain
\begin{align*}
AG% \geq  & \frac{2}{n}\phi G^{2} + \bigg( -\frac{2}{n}a\omega^{p}-\frac{a}{2}p\omega^{p-1}+\frac{3a}{2} p(p-1)\omega^{p-2}\bigg)\phi G\\
%& +\bigg( \frac{a^{2}}{2n}\omega^{2p}-\frac{3}{4}a^{2}p\omega^{2p-1}-\frac{9a^{2}}{4}p(p-1)\omega^{2p-2}+3Ka\omega^{p}\bigg)\phi\\
%& -2K\phi G-\frac{R\left\vert \nabla\phi \right\vert^{2}}{C_{1}\phi^{\frac{1}{2}}} G^{\frac{3}{2}}-\frac{C_{1}\phi^{\frac{1}{2}}G^{\frac{1}{2}}}{R}\bigg(G-\frac{3a}{2} \omega^{p}\bigg)\\
\geq\, & \frac{2}{n}\phi G^{2} -\bigg(2K+\frac{12\left\vert b\right\vert}{n}\bigg)\phi G-\frac{R\left\vert \nabla\phi \right\vert^{2}}{C_{1}\phi^{\frac{1}{2}}} G^{\frac{3}{2}}-\frac{C_{1}\phi^{\frac{1}{2}}G^{\frac{3}{2}}}{R}\\
& +\bigg( -\frac{4(\lambda-1)a}{n}\omega^{p}+(\lambda-2)ap\omega^{p-1}+\lambda a p(p-1)\omega^{p-2}+\frac{2(\lambda-1)a\omega^{p-2}}{n\left\vert J\right\vert^{2}}\bigg)\phi G\\
& + \bigg( -\frac{2(\lambda-1)}{n\left\vert J\right\vert^{2}}\phi G +\frac{\lambda C_{1}}{R}\omega^{2}\phi^{\frac{1}{2}}G^{\frac{1}{2}}+\bigg(\frac{2(\lambda-1)^{2}a}{n}\omega^{p+2}-\lambda(\lambda-1)ap\omega^{p+1}\\
& -\lambda^{2}ap(p-1)\omega^{p}+\bigg(2\lambda K+\frac{4(\lambda-1)}{n}(2\left\vert b\right\vert-b)\bigg)\omega^{2}
+(2(2-\lambda)\left\vert b\right\vert-\lambda b)p\omega\\
& -2\lambda \left\vert b\right\vert p(p-1)\bigg)\phi\bigg)a\omega^{p-2}\\
\geq\, & \frac{2}{n}\phi G^{2} -\bigg(2K+\frac{12\left\vert b\right\vert}{n}\bigg)\phi G-\frac{2C_{1}\phi^{\frac{1}{2}}G^{\frac{3}{2}}}{R}-\bigg( \frac{(2-\lambda)ap}{\left\vert J \right\vert^{p-1}}+\frac{\lambda ap(p-1)}{\left\vert J \right\vert^{p-2}}\\
& +\frac{2(\lambda-1)a}{n\left\vert J\right\vert^{p}}\bigg)\phi G+ \bigg( -\frac{2(\lambda-1)}{n\left\vert J\right\vert^{2}}\phi G +\frac{\lambda C_{1}}{R\left\vert J \right\vert^{2}}\phi^{\frac{1}{2}}G^{\frac{1}{2}}+\bigg(\frac{\lambda^{2}ap(p-1)}{\left\vert J \right\vert^{p}}\\
& +\frac{2\lambda nK+12(\lambda-1)\left\vert b\right\vert}{n\left\vert J\right\vert^{2}} +\frac{(4-\lambda)p\left\vert b\right\vert}{ \left\vert J\right\vert}\bigg)\phi\bigg)a\omega^{p-2}.
\end{align*}

If
\begin{align*}
-\frac{2(\lambda-1)}{n\left\vert J\right\vert^{2}}\phi G +\frac{\lambda C_{1}}{R\left\vert J \right\vert^{2}}\phi^{\frac{1}{2}}G^{\frac{1}{2}}+\bigg(\frac{\lambda^{2}ap(p-1)}{\left\vert J \right\vert^{p}} +\frac{2\lambda K}{\left\vert J\right\vert^{2}} +\frac{12(\lambda-1)\left\vert b\right\vert}{n\left\vert J\right\vert^{2}}+\frac{(4-\lambda)p\left\vert b\right\vert}{ \left\vert J\right\vert}\bigg)\phi \leq 0,
\end{align*}
then
\begin{equation}\label{w}
\begin{split}
AG  \geq & \frac{2}{n}\phi G^{2} -\bigg(2K+\frac{12\left\vert b\right\vert}{n}\bigg)\phi G-\frac{2C_{1}\phi^{\frac{1}{2}}G^{\frac{3}{2}}}{R}-\bigg( \frac{(2-\lambda)ap}{\left\vert J \right\vert^{p-1}}+\frac{\lambda ap(p-1)}{\left\vert J \right\vert^{p-2}}\\
& +\frac{2(\lambda-1)a}{n\left\vert J\right\vert^{p}}\bigg)\phi G.
\end{split}
\end{equation}
Noting that
\[
\left\vert J\right\vert \leq \frac{2(\lambda-1)}{(2-\lambda)np} \quad\quad \mbox{and} \quad\quad \left\vert J\right\vert^{2} \leq \frac{2(\lambda-1)}{\lambda np(p-1)}.
\]
Thus, we can see easily from \eqref{w} and the Young's inequality that
\begin{equation}\label{17}
\phi G \leq n\left(2K + A +\frac{nC_{1}^{2}}{R^{2}}+\frac{12\left\vert b\right\vert}{n}+\frac{6(\lambda-1)a}{n\left\vert J\right\vert^{p}}\right).
\end{equation}

If
\begin{align*}
 -\frac{2(\lambda-1)}{n\left\vert J\right\vert^{2}}\phi G +\frac{\lambda C_{1}}{R\left\vert J \right\vert^{2}}\phi^{\frac{1}{2}}G^{\frac{1}{2}}+\bigg(\frac{\lambda^{2}ap(p-1)}{\left\vert J \right\vert^{p}} +\frac{2\lambda K}{\left\vert J\right\vert^{2}} +\frac{12(\lambda-1)\left\vert b\right\vert}{n\left\vert J\right\vert^{2}} +\frac{(4-\lambda)p\left\vert b\right\vert}{ \left\vert J\right\vert}\bigg)\phi > 0.
\end{align*}
Then, it is easy to see that
\begin{equation}\label{18}
\phi G \leq n\left( \frac{2\lambda}{\lambda-1}K+\frac{n\lambda^{2}C_{1}^{2}}{4(\lambda-1)^{2}R^{2}}+\frac{2\lambda a}{n\left\vert J \right\vert^{p}}+\frac{2(16-7\lambda)\left\vert b\right\vert}{(\lambda-1)n}\right).
\end{equation}
Hence, by combining \eqref{17} and \eqref{18}, we conclude
\begin{align}\label{19}
\sup_{B_R}\bigg(\frac{\left\vert \nabla u \right\vert^{2}}{u^{2}}+\lambda a (\log u)^{p}+2\left\vert b\right\vert \bigg) \leq \,\, & n\bigg( \frac{2\lambda}{\lambda-1}K+A+\frac{n\lambda^{2}C_{1}^{2}}{4(\lambda-1)^{2}R^{2}}+\frac{\max\{6(\lambda-1),2\lambda\}a}{n\left\vert J \right\vert^{p}}\nonumber\\
&\quad +\max\bigg\{6,\frac{16-7\lambda}{\lambda-1}\bigg\}\frac{2\left\vert b\right\vert}{n}\bigg).
\end{align}

Now, by combining \eqref{20}, \eqref{58}, \eqref{31} and \eqref{19}, we complete the proof of Theorem \ref{main} for the case $a>0$.
\medskip

{\bf Case $2$}: $a < 0$. For the present situation, we need to consider the following four cases on $\omega$: $(1).\, \omega^{p} < 0 $ and $\omega \leq L$ ; $(2). \, \omega^{p} < 0 $ and $L < \omega < 0$; $(3). \, \omega^{p} \geq 0 $ and $\omega \in(-\infty, \frac{1}{J})\cup(V, +\infty)$ and $(4). \, \omega^{p} \geq 0 $ and $\frac{1}{J} \leq \omega \leq V$. We will discuss them one by one.

\medskip

$(1).$  If $\omega^{p} < 0 $ and $\omega \leq L$, we have the followings
\begin{equation}\label{52}
-\frac{4(\lambda-1)a}{n}\omega^{p}+(\lambda-2)ap\omega^{p-1}+\lambda ap(p-1)\omega^{p-2} \geq -\frac{4(\lambda-1)}{\lambda n}G,
\end{equation}

\begin{equation}\label{52'}
\frac{2(\lambda-1)^{2}}{n}\omega^{2} -\lambda(\lambda-1) p\omega-\lambda^{2} p(p-1) \geq 0,
\end{equation}
and
\begin{equation}\label{52"}
\begin{split}
\left[\left(2\lambda K+\frac{4(\lambda-1)}{n}(2\left\vert b \right\vert -b)\right)\omega^{2}+(2(2-\lambda)\left\vert b\right\vert-\lambda b)p\omega -2\lambda \left\vert b \right\vert p(p-1)\right]a\omega^{p-2} \geq 0.
\end{split}
\end{equation}
By substituting \eqref{52}, \eqref{52'} and \eqref{52"} into \eqref{0}, we derive
\begin{align*}
AG \geq\, & \frac{2}{n}\phi G^{2} -\bigg(\frac{4(\lambda-1)}{\lambda n}G+2K+\frac{12\left\vert b\right\vert}{n}\bigg)\phi G-2\left\vert \nabla\phi \right\vert G^{\frac{3}{2}}\\
& + \bigg(\frac{2(\lambda-1)^{2}}{n}\omega^{2} -\lambda(\lambda-1) p\omega-\lambda^{2} p(p-1)\bigg)a^{2}\omega^{2p-2}\phi\\
& +\left(\left(2\lambda K+\frac{4(\lambda-1)}{n}(2\left\vert b \right\vert -b)\right)\omega^{2}+(2(2-\lambda)\left\vert b\right\vert-\lambda b)p\omega -2\lambda \left\vert b \right\vert p(p-1)\right)a\omega^{p-2}\phi\\
\geq\,&\frac{2(2-\lambda)}{\lambda n}\phi G^{2} -\bigg(2K+\frac{12\left\vert b\right\vert}{n}\bigg)\phi G-2\left\vert \nabla\phi \right\vert G^{\frac{3}{2}}.
\end{align*}
Dividing both side of the last inequality by $G$, we obtain
\begin{align}\label{10}
A \geq \frac{2(2-\lambda)}{\lambda n}Q  -\left(2K+\frac{12\left\vert b \right\vert}{n}\right)-2\frac{C_{1}}{R}Q^{\frac{1}{2}}.
\end{align}
Then it follows easily from Young's inequality that
\[
Q \leq \frac{\lambda n}{2-\lambda}\left(2K+A+\frac{12\left\vert b \right\vert}{n}+\frac{\lambda nC_{1}^{2}}{(2-\lambda)R^{2}}\right).
\]
Hence,
\begin{equation}\label{30}
\sup_{\mathbb{B}_{R}}G \leq \sup_{\mathbb{B}_{2R}}(\phi G)\leq \frac{\lambda n}{2-\lambda}\left(2K+A+\frac{12\left\vert b \right\vert}{n}+\frac{\lambda nC_{1}^{2}}{(2-\lambda)R^{2}}\right).
\end{equation}

\medskip

$(2)$. If $\omega^{p} < 0$ and $L < \omega < 0$, then we have
\begin{equation}\label{28}
\begin{split}
-\frac{4(\lambda-1)a}{n} \omega^{p}+ (\lambda-2) ap\omega^{p-1}+\frac{\lambda a}{3} p(p-1)\omega^{p-2} \geq -\frac{4(\lambda-1)}{\lambda n}G,\\
\end{split}
\end{equation}

\begin{equation}\label{28'}
\begin{split}
\frac{2(\lambda-1)^{2}}{n}a\omega^{p+2}-\lambda(\lambda-1) ap\omega^{p+1}-\lambda^{2}ap(p-1)\omega^{p} \geq \lambda^{2}ap(p-1)\left\vert L \right\vert^{p},\\
\end{split}
\end{equation}
and
\begin{equation}\label{28"}
\begin{split}
&\left(2\lambda K+\frac{4(\lambda-1)}{n}\left(2\left\vert b\right\vert-b\right)\right)\omega^{2}+(2(2-\lambda)\left\vert b\right\vert-\lambda b)p\omega-2\lambda \left\vert b\right\vert p(p-1)\\
\geq &-(4-\lambda)p\left\vert b\right\vert \left\vert L\right\vert-2\lambda \left\vert b\right\vert p(p-1).
\end{split}
\end{equation}
Thus, by substituting \eqref{28}, \eqref{28'} and \eqref{28"} into \eqref{0}, we obtain
\begin{align*}
AG \geq\,& \frac{2}{n}\phi G^{2} -\bigg(2K+\frac{12\left\vert b\right\vert}{n}\bigg)\phi G-2\left\vert \nabla\phi \right\vert G^{\frac{3}{2}} + \bigg(-\frac{4(\lambda-1)a}{n} \omega^{p}+ (\lambda-2) ap\omega^{p-1}\\
& +\frac{\lambda a}{3} p(p-1)\omega^{p-2}\bigg)\phi G+\bigg[\frac{2\lambda}{3}p(p-1)\phi G+ \bigg(\frac{2(\lambda-1)^{2}}{n}a\omega^{p+2}-\lambda(\lambda-1) ap\omega^{p+1}\\
& -\lambda^{2}ap(p-1)\omega^{p} +\left(2\lambda K+\frac{4(\lambda-1)}{n}\left(2\left\vert b\right\vert-b\right)\right)\omega^{2}+(2(2-\lambda)\left\vert b\right\vert-\lambda b)p\omega\\
& -2\lambda \left\vert b\right\vert p(p-1)\bigg)\phi \bigg] a\omega^{p-2}\\
\geq\,& \frac{2(2-\lambda)}{\lambda n}\phi G^{2} -\bigg(2K+\frac{12\left\vert b\right\vert}{n}\bigg)\phi G-2\left\vert \nabla\phi \right\vert G^{\frac{3}{2}}+\bigg(\frac{2\lambda}{3}p(p-1)\phi G+ \lambda^{2}ap(p-1)\left\vert L \right\vert^{p}\\
& -(4-\lambda)p\left\vert b\right\vert \left\vert L\right\vert-2\lambda \left\vert b\right\vert p(p-1)\bigg)a\omega^{p-2}.
\end{align*}

If $$  \frac{2\lambda}{3}p(p-1)\phi G+ \lambda^{2}ap(p-1)\left\vert L \right\vert^{p} -(4-\lambda)p\left\vert b\right\vert \left\vert L\right\vert-2\lambda \left\vert b\right\vert p(p-1) \geq 0,$$
then,
\begin{equation}
\begin{split}
AG & \geq \frac{2(2-\lambda)}{\lambda n}\phi G^{2} -\bigg(2K+\frac{12\left\vert b\right\vert}{n}\bigg)\phi G-2\left\vert \nabla\phi \right\vert G^{\frac{3}{2}}\\
& \geq \frac{2(2-\lambda)}{\lambda n}\phi G^{2} -\bigg(2K+\frac{12\left\vert b\right\vert}{n}\bigg)\phi G-2\frac{C_{1}}{R}\phi^{\frac{1}{2}} G^{\frac{3}{2}}.
\end{split}
\end{equation}
Hence, it follows immediately from Young's inequality that
\begin{equation}\label{43}
\phi G \leq \frac{\lambda n}{2-\lambda}\left(2K+A+\frac{12\left\vert b\right\vert}{n}+\frac{\lambda nC_{1}^{2}}{(2-\lambda)R^{2}}\right).
\end{equation}

If $$\frac{2\lambda}{3}p(p-1)\phi G+ \lambda^{2}ap(p-1)\left\vert L \right\vert^{p} -(4-\lambda)p\left\vert b\right\vert \left\vert L\right\vert-2\lambda \left\vert b\right\vert p(p-1) <0,$$
then it is easy to see that
\begin{equation}\label{44}
\phi G \leq -\frac{3\lambda \left\vert L \right\vert^{p}}{2}a+\frac{3(4-\lambda)\left\vert b\right\vert \left\vert L\right\vert}{2\lambda(p-1)}+3\left\vert b\right\vert.
\end{equation}
Combining \eqref{43} and \eqref{44} we get
\begin{equation}\label{45}
\begin{split}
\sup_{B_R}\left(\frac{\left\vert \nabla u \right\vert^{2}}{u^{2}}+\lambda a (\log u)^{p} +2\left\vert b\right\vert \right)\leq \,\, & n\bigg(\frac{2\lambda}{2-\lambda}K + \frac{\lambda}{2-\lambda}A +\frac{12\lambda \left\vert b\right\vert}{(2-\lambda)n}+\frac{\lambda^{2} nC_{1}^{2}}{(2-\lambda)^{2}R^{2}}\\
& -\frac{3\lambda\left\vert L \right\vert^{p}}{2n}a+\frac{3(4-\lambda)\left\vert b\right\vert \left\vert L\right\vert}{2\lambda n(p-1)}\bigg).
\end{split}
\end{equation}

\medskip

$(3)$. If $ \omega^{p} \geq 0 $ and $\omega \in(-\infty, \frac{1}{J})\cup(V, +\infty)$, then
\begin{equation}\label{46}
\begin{split}
-\frac{4(\lambda-1)}{n}+(\lambda-2)p\omega^{-1}+\lambda p(p-1)\omega^{-2} \leq -\frac{2(\lambda-1)}{n},\\
\end{split}
\end{equation}
\begin{equation}\label{46'}
\begin{split}
\frac{2(\lambda-1)^{2}}{n}-\lambda(\lambda-1) p\omega^{-1}-\lambda^{2} p(p-1)\omega^{-2} \geq  0,
\end{split}
\end{equation}
and
\begin{equation}\label{46"}
\begin{split}
&2\lambda K+\frac{4(\lambda-1)}{n}\left(2\left\vert b\right\vert-b\right)+(2(2-\lambda)\left\vert b \right\vert-\lambda b)p\omega^{-1}-2\lambda\left\vert b\right\vert p(p-1)\omega^{-2} \\
\leq & 2\lambda K+\frac{12(\lambda-1)}{n}\left\vert b\right\vert+(4-\lambda)p\left\vert b\right\vert \left\vert J\right\vert.
\end{split}
\end{equation}
By substituting \eqref{22}, \eqref{46}, \eqref{46'} and \eqref{46"} into \eqref{0}, we derive
\begin{align*}
AG \geq\, &  \frac{2}{n}\phi G^{2} -\bigg(2K+\frac{12\left\vert b\right\vert}{n}\bigg)\phi G-\frac{R\left\vert \nabla\phi \right\vert^{2}}{C_{1}\phi^{\frac{1}{2}}} G^{\frac{3}{2}}-\frac{C_{1}\phi^{\frac{1}{2}}G^{\frac{3}{2}}}{R}\\
&+ \bigg[\bigg( -\frac{4(\lambda-1)}{n}+(\lambda-2)p\omega^{-1}+\lambda p(p-1)\omega^{-2}\bigg)\phi G+\frac{\lambda C_{1}}{R}\phi^{\frac{1}{2}}G^{\frac{1}{2}}\\
& +\bigg(\frac{2(\lambda-1)^{2}}{n}-\lambda(\lambda-1) p\omega^{-1}-\lambda^{2} p(p-1)\omega^{-2}\bigg)a\omega^{p}\phi+\bigg(2\lambda K\\
& +\frac{4(\lambda-1)}{n}\left(2\left\vert b\right\vert-b\right) +(2(2-\lambda)\left\vert b \right\vert-\lambda b)p\omega^{-1}-2\lambda\left\vert b\right\vert p(p-1)\omega^{-2}\bigg)\phi\bigg]a\omega^{p}\\
\geq\, & \frac{2}{n}\phi G^{2} -\bigg(2K+\frac{12\left\vert b\right\vert}{n}\bigg)\phi G-\frac{R\left\vert \nabla\phi \right\vert^{2}}{C_{1}\phi^{\frac{1}{2}}} G^{\frac{3}{2}}-\frac{C_{1}\phi^{\frac{1}{2}}G^{\frac{3}{2}}}{R}+\bigg(-\frac{2(\lambda-1)}{n}\phi G\\
& +\frac{\lambda C_{1}}{R}\phi^{\frac{1}{2}}G^{\frac{1}{2}}+2\lambda K+\frac{12(\lambda-1)}{n}\left\vert b\right\vert+(4-\lambda)p\left\vert b\right\vert \left\vert J\right\vert\bigg)a\omega^{p}\\
\geq\, & \frac{2}{n}\phi G^{2} -\bigg(2K+\frac{12\left\vert b\right\vert}{n}\bigg)\phi G-\frac{2C_{1}}{R}\phi^{\frac{1}{2}}G^{\frac{3}{2}}+\bigg(-\frac{2(\lambda-1)}{n}\phi G+\frac{\lambda C_{1}}{R}\phi^{\frac{1}{2}}G^{\frac{1}{2}}\\
& +2\lambda K+\frac{12(\lambda-1)}{n}\left\vert b\right\vert+(4-\lambda)p\left\vert b\right\vert \left\vert J\right\vert\bigg)a\omega^{p}.
\end{align*}

If $$  -\frac{2(\lambda-1)}{n}\phi G+\frac{\lambda C_{1}}{R}\phi^{\frac{1}{2}}G^{\frac{1}{2}} +2\lambda K+\frac{12(\lambda-1)}{n}\left\vert b\right\vert+(4-\lambda)p\left\vert b\right\vert \left\vert J\right\vert \leq 0,$$
then
\begin{equation}
\begin{split}
AG & \geq \frac{2}{n}\phi G^{2} -\bigg(2K+\frac{12\left\vert b\right\vert}{n}\bigg)\phi G-\frac{2C_{1}}{R}\phi^{\frac{1}{2}}G^{\frac{3}{2}}.
\end{split}
\end{equation}
Immediately, it follows from Young's inequality that
\begin{equation}\label{63}
\phi G \leq n\left(2K+A+\frac{12\left\vert b\right\vert}{n}+\frac{nC_{1}^{2}}{R^{2}}\right).
\end{equation}

If $$-\frac{2(\lambda-1)}{n}\phi G+\frac{\lambda C_{1}}{R}\phi^{\frac{1}{2}}G^{\frac{1}{2}} +2\lambda K+\frac{12(\lambda-1)}{n}\left\vert b\right\vert+(4-\lambda)p\left\vert b\right\vert \left\vert J\right\vert >0,$$
then it is easy to see that
\begin{equation}\label{64}
\phi G \leq n\left(\frac{2\lambda}{\lambda-1}K+\frac{n\lambda^{2}C_{1}^{2}}{4(\lambda-1)^{2}R^{2}}+\frac{12}{n}\left\vert b\right\vert+\frac{4-\lambda}{\lambda-1}\left\vert b\right\vert p\left\vert J\right\vert\right).
\end{equation}
Thus, in this case, by combining \eqref{63} and \eqref{64}, we get
\begin{equation}\label{57}
\sup_{B_R}\left(\frac{\left\vert \nabla u \right\vert^{2}}{u^{2}}+\lambda a (\log u)^{p} +2\left\vert b\right\vert \right)\leq n\left(\frac{2\lambda}{\lambda-1}K+A+\frac{n\lambda^{2}C_{1}^{2}}{4(\lambda-1)^{2}R^{2}}+\frac{2(16-7\lambda)}{(2-\lambda)n}\left\vert b\right\vert\right).
\end{equation}
\medskip

$(4)$. If $ \omega^{p} \geq 0 $ and $\frac{1}{J} \leq \omega \leq V$, we obtain
\begin{equation}\label{56}
\begin{split}
&-\frac{4(\lambda-1)a}{n}\omega^{p}+(\lambda-2)ap\omega^{p-1}+\lambda a p(p-1)\omega^{p-2} +\frac{2(\lambda-1)aV^{2}}{n}\omega^{p-2} \\
\geq \,\, & \frac{2(\lambda-1)aV^{p}}{n }+(2-\lambda)apV^{p-1}+\lambda a p(p-1)V^{p-2},
\end{split}
\end{equation}
\begin{equation}\label{56'}
\begin{split}
&\frac{2(\lambda-1)^{2}a}{n}\omega^{p+2}-\lambda(\lambda-1)ap\omega^{p+1}-\lambda^{2}ap(p-1)\omega^{p}\\
 \leq \,\, & -\lambda(\lambda-1)apV^{p+1}-\lambda^{2}p(p-1)aV^{p},
\end{split}
\end{equation}
and
\begin{equation}\label{56"}
\begin{split}
&\bigg(2\lambda K+\frac{4(\lambda-1)}{n}\left(2\left\vert b\right\vert-b\right)\bigg)\omega^{2}+(2(2-\lambda)\left\vert b \right\vert-\lambda b)p\omega-2\lambda\left\vert b\right\vert p(p-1)\\
 \leq \,\, & 2\lambda KV^{2}+\frac{12(\lambda-1)\left\vert b\right\vert V^{2}}{n}+(4-\lambda)\left\vert b\right\vert pV.
\end{split}
\end{equation}
By substituting \eqref{22}, \eqref{56}, \eqref{56'} and \eqref{56"} into \eqref{0}, we derive
\begin{align*}
AG \geq \,\, &  \frac{2}{n}\phi G^{2} -\bigg(2K+\frac{12\left\vert b\right\vert}{n}\bigg)\phi G-\frac{R\left\vert \nabla\phi \right\vert^{2}}{C_{1}\phi^{\frac{1}{2}}} G^{\frac{3}{2}}-\frac{C_{1}\phi^{\frac{1}{2}}G^{\frac{3}{2}}}{R}\\
& +\bigg( -\frac{4(\lambda-1)a}{n}\omega^{p}+(\lambda-2)ap\omega^{p-1}+\lambda a p(p-1)\omega^{p-2} +\frac{2(\lambda-1)aV^{2}}{n}\omega^{p-2}\bigg)\phi G\\
& + \bigg[-\frac{2(\lambda-1)V^{2}}{n}\phi G+\frac{\lambda C_{1}}{R}\omega^{2}\phi^{\frac{1}{2}}G^{\frac{1}{2}}+\bigg(\frac{2(\lambda-1)^{2}a}{n}\omega^{p+2}-\lambda(\lambda-1)ap\omega^{p+1}\\
& -\lambda^{2}ap(p-1)\omega^{p}\bigg)\phi +\bigg(\bigg(2\lambda K+\frac{4(\lambda-1)}{n}\left(2\left\vert b\right\vert-b\right)\bigg)\omega^{2}+(2(2-\lambda)\left\vert b \right\vert-\lambda b)p\omega\\
& -2\lambda\left\vert b\right\vert p(p-1)\bigg)\phi\bigg]a\omega^{p-2}\\
\geq \,\, & \frac{2}{n}\phi G^{2} -\bigg(2K+\frac{12\left\vert b\right\vert}{n}\bigg)\phi G-2\frac{C_{1}\phi^{\frac{1}{2}}G^{\frac{3}{2}}}{R} +\bigg( \frac{2(\lambda-1)aV^{p}}{n }+(2-\lambda)apV^{p-1}\\
& +\lambda a p(p-1)V^{p-2}\bigg)\phi G+\bigg(-\frac{2(\lambda-1)V^{2}}{n}\phi G+\frac{\lambda C_{1}V^{2}}{R}\phi^{\frac{1}{2}}G^{\frac{1}{2}} -\lambda(\lambda-1)apV^{p+1}\\
& -\lambda^{2}p(p-1)aV^{p}+2\lambda KV^{2}+\frac{12(\lambda-1)\left\vert b\right\vert V^{2}}{n}+(4-\lambda)\left\vert b\right\vert pV\bigg)a\omega^{p-2}.
\end{align*}
\medskip

We need to discuss the following two cases. If
\begin{align*}
&-\frac{2(\lambda-1)V^{2}}{n}\phi G+\frac{\lambda C_{1}V^{2}}{R}\phi^{\frac{1}{2}}G^{\frac{1}{2}} -\lambda(\lambda-1)apV^{p+1} -\lambda^{2}p(p-1)aV^{p}\\&+2\lambda KV^{2}+\frac{12(\lambda-1)\left\vert b\right\vert V^{2}}{n}+(4-\lambda)\left\vert b\right\vert pV \leq 0,
\end{align*}
then
\begin{equation}
\begin{split}
AG  \geq \,\, &\frac{2}{n}\phi G^{2} -\bigg(2K+\frac{12\left\vert b\right\vert}{n}\bigg)\phi G-2\frac{C_{1}\phi^{\frac{1}{2}}G^{\frac{3}{2}}}{R}+\bigg( \frac{2(\lambda-1)aV^{p}}{n }+(2-\lambda)apV^{p-1}\\
&+\lambda a p(p-1)V^{p-2}\bigg)\phi G.
\end{split}
\end{equation}
Noting that,
\begin{align*}
\frac{1}{V} & =\frac{4}{n\left(p+\sqrt{p^{2}+\frac{8}{n}p(p-1)}\right)} \frac{\lambda-1}{\lambda} \leq \frac{2(\lambda-1)}{\lambda np},
\end{align*}
and
\begin{align*}
\frac{1}{V} & \leq \frac{2(\lambda-1)}{\sqrt{2np(p-1)}\lambda}.
\end{align*}
We can deduce easily from the above and the Young's inequality that
\begin{equation}\label{53}
\phi G \leq n\left(2K+A+\frac{12\left\vert b\right\vert}{n}+\frac{nC_{1}^{2}}{R^{2}}-\frac{(2\lambda^{2}-2)aV^{p}}{\lambda n}\right).
\end{equation}
\medskip

Now, if
\begin{align*}
&-\frac{2(\lambda-1)V^{2}}{n}\phi G+\frac{\lambda C_{1}V^{2}}{R}\phi^{\frac{1}{2}}G^{\frac{1}{2}} -\lambda(\lambda-1)apV^{p+1} -\lambda^{2}p(p-1)aV^{p}\\&+2\lambda KV^{2}+\frac{12(\lambda-1)\left\vert b\right\vert V^{2}}{n}+(4-\lambda)\left\vert b\right\vert pV > 0.
\end{align*}
Then, it is easy to see that
\begin{equation}\label{54}
\phi G \leq n\left(\frac{2\lambda}{\lambda-1}K+\frac{\lambda^{2}nC_{1}^{2}}{4(\lambda-1)^{2}R^{2}}+\frac{2(5\lambda+4)\left\vert b\right\vert}{\lambda n}-\frac{4(\lambda-1)aV^{p}}{n}\right).
\end{equation}
Thus, by combining \eqref{53} and \eqref{54}, we arrive at
\begin{equation}\label{55}
\begin{split}
\sup_{B_R}\left(\frac{\left\vert \nabla u \right\vert^{2}}{u^{2}}+\lambda a (\log u)^{p} +2\left\vert b\right\vert \right)\leq \,\, & n\bigg(\frac{2\lambda}{\lambda-1}K+A+\frac{n\lambda^{2}C_{1}^{2}}{4(\lambda-1)^{2}R^{2}}+\frac{2(5\lambda+4)\left\vert b\right\vert}{\lambda n}\\
& -\frac{4(\lambda-1)aV^{p}}{n}\bigg).
\end{split}
\end{equation}
Hence, by summarizing the previous \eqref{30}, \eqref{45}, \eqref{57} and \eqref{55}, we complete the proof of Theorem \ref{main}.
\end{proof}

\begin{proof}[\bf Proof of Theorem \ref{any}]
For the case $a>0$, by taking $\lambda=\frac{3}{2}$ in Theorem \ref{main}, we obtain
\begin{align*}
&\frac{\left\vert \nabla u \right\vert^{2}}{u^{2}} +\frac{3a}{2}(\log u)^{p}\\
\leq & 3n \bigg( 2K+A+\frac{6nC_{1}^{2}}{R^{2}}+\frac{(2n(p-1))^{p-1}}{2}a+\frac{12\left\vert b\right\vert}{n}+\frac{a}{n\left\vert J \right\vert^{p}}+\frac{aV^{p+1}+V\left\vert b\right\vert}{n(p-1)}\bigg).
\end{align*}
Here, we have used the fact $\left\vert L\right\vert \leq V$ for $\lambda =\frac{3}{2}$, which can be checked by a direct calculation. For the case $\lambda\geq 2$, multiplying the above inequality by $\frac{2\lambda}{3}$ leads to
\begin{equation*}
\begin{aligned}
&\frac{2\lambda}{3}\frac{\left\vert \nabla u \right\vert^{2}}{u^{2}} +\lambda a(\log u)^{p}\\
\leq\, &2\lambda n \bigg( 2K+A+\frac{6nC_{1}^{2}}{R^{2}}+\frac{(2n(p-1))^{p-1}}{2}a+\frac{12\left\vert b\right\vert}{n}+\frac{a}{n\left\vert J \right\vert^{p}}+\frac{aV^{p+1}+V\left\vert b\right\vert}{n(p-1)}\bigg).
\end{aligned}
\end{equation*}
Since the fact $\lambda\geq 2$ implies
$$\frac{\left\vert \nabla u \right\vert^{2}}{u^{2}} +\lambda a(\log u)^{p}\leq \frac{2\lambda}{3}\frac{\left\vert \nabla u \right\vert^{2}}{u^{2}} +\lambda a(logu)^{p},$$
it follows that
\[
\frac{\left\vert \nabla u \right\vert^{2}}{u^{2}} +\lambda a(\log u)^{p}\leq 2\lambda n \bigg( 2K+A+\frac{6nC_{1}^{2}}{R^{2}}+\frac{(2n(p-1))^{p-1}}{2}a+\frac{12\left\vert b\right\vert}{n}+\frac{a}{n\left\vert J \right\vert^{p}}+\frac{V(aV^{p}+\left\vert b\right\vert)}{n(p-1)}\bigg).
\]

For the case $a<0$, we can use the same method to get the desired result. Thus, we finish the proof.
\end{proof}

\begin{proof}[\bf Proof of Corollary \ref{har}]
In the case $a>0$ and $p=\frac{2k}{2k_{2}+1} \geq 2$ , let $\lambda=\frac{3}{2}$, then it is easy to see that
$$\frac{3a}{2} \big(\log u\big)^{p}\ge 0.$$
As $\lambda =\frac{3}{2}$, we can check easily $\left\vert L\right\vert \leq V$. Theorem \ref{main} tells us that there holds true
\begin{align*}
\frac{\left\vert \nabla u \right\vert^{2}}{u^{2}} \leq &3n \bigg( 2K+A+\frac{6nC_{1}^{2}}{R^{2}}+\frac{(2n(p-1))^{p-1}}{2}a+\frac{12\left\vert b\right\vert}{n}+\frac{a}{n\left\vert J \right\vert^{p}}+\frac{V^{p+1}a}{n(p-1)}+\frac{V\left\vert b\right\vert}{n(p-1)} \bigg).
\end{align*}
Choose $x,y \in \mathbb{B}_{R/2}(O)$ such that
$$u(x)=\sup_{\mathbb{B}_{R/2}(O)}u(x)\ \ \text{and}\ \ u(y)=\inf_{\mathbb{B}_{R/2}(O)}u(x).$$
Let $\gamma(t), t\in [0,l]$ be a shortest curve with arc length in $(M,g)$ connecting $y$ and $x$ with $\gamma(0)=x,\ \gamma(l)=y$. By the triangle inequality, we can see easily that $\gamma\in \mathbb{B}_{R}(O)$ and $l\leq R$. Then

\begin{align*}
&\log u(x) -\log u(y) \leq \int_{\gamma} \ \frac{\left\vert \nabla u \right\vert}{u} \,\\
\leq \,& \int_{\gamma} \ \bigg[ 3n \bigg( 2K+A+\frac{6nC_{1}^{2}}{R^{2}}+\frac{(2n(p-1))^{p-1}}{2}a+\frac{12\left\vert b\right\vert}{n}+\frac{a}{n\left\vert J \right\vert^{p}}+\frac{V^{p+1}a}{n(p-1)}+\frac{V\left\vert b\right\vert}{n(p-1)} \bigg)\bigg]^{\frac{1}{2}} \,\\
\leq \, & \bigg[ 3n \bigg( 2K+A+\frac{6nC_{1}^{2}}{R^{2}}+\frac{(2n(p-1))^{p-1}}{2}a+\frac{12\left\vert b\right\vert}{n}+\frac{a}{n\left\vert J \right\vert^{p}}+\frac{V^{p+1}a}{n(p-1)}+\frac{V\left\vert b\right\vert}{n(p-1)} \bigg)\bigg]^{\frac{1}{2}} \cdot R.
\end{align*}
Hence
\[
\sup_{\mathbb{B}_{R/2}(O)}u \leq e^{R\sqrt{3n \left( 2K+A+\frac{6nC_{1}^{2}}{R^{2}}+\frac{(2n(p-1))^{p-1}}{2}a+\frac{12\left\vert b\right\vert}{n}+\frac{a}{n\left\vert J \right\vert^{p}}+\frac{V^{p+1}a}{n(p-1)}+\frac{V\left\vert b\right\vert}{n(p-1)} \right)}}\inf_{\mathbb{B}_{R/2}(O)}u.
\]
We finish the proof.
\end{proof}
\begin{proof}[\bf Proof of Corollary \ref{mor}:] By letting $\lambda=\frac{3}{2}$ in Theorem \ref{main}, we can see easily that
\begin{align*}
\big(\log u\big)^{p}\leq &\frac{2n}{a} \left( 2K+A+\frac{6nC_{1}^{2}}{R^{2}}+\frac{12\left\vert b\right\vert}{n}+\frac{V\left\vert b\right\vert}{n(p-1)} \right) + \left(n^p(2(p-1))^{p-1}+\frac{2}{\left\vert J \right\vert^{p}}+\frac{2V^{p+1}}{(p-1)}\right).
\end{align*}
Then, by letting $R\rightarrow+\infty$ in the above inequality we obtain
\begin{align*}
\big(\log u\big)^{p}\leq &\frac{2n}{a} \left( 2K +\frac{12\left\vert b\right\vert}{n}+\frac{V\left\vert b\right\vert}{n(p-1)} \right) + \left(n^p(2(p-1))^{p-1}+\frac{2}{\left\vert J \right\vert^{p}}+\frac{2V^{p+1}}{(p-1)}\right).
\end{align*}
Other required inequalities on a priori estimates on the bound of $u$ follows immediately.
\end{proof}

\medskip

\noindent {\it\textbf{Acknowledgements}}: The authors are supported partially by NSFC grant (No.11731001). The author Y. Wang is supported partially by NSFC grant (No.11971400) and Guangdong Basic and Applied Basic Research Foundation Grant (No. 2020A1515011019).

\bibliographystyle{amsalpha}

\end{document}